\documentclass{amsart}
\usepackage[dvips]{graphicx}
\usepackage{amscd}
\usepackage{pstricks}

\theoremstyle{plain}
\newtheorem{theorem}{Theorem}
\newtheorem{proposition}{Proposition}
\newtheorem{lemma}{Lemma}
\newtheorem{definition}{Definition}

\theoremstyle{remark}
\newtheorem{remark}{Remark}

\numberwithin{equation}{section}

\DeclareMathOperator{\diag}{Diag}

\DeclareMathOperator{\rank}{Rank}

\begin{document}

\title[]{Almost reducibility of linear difference systems from a spectral point of view}
\author[]{\'Alvaro Casta\~neda}
\author[]{Gonzalo Robledo}
\address{Departamento de Matem\'aticas, Facultad de Ciencias, Universidad de
  Chile, Casilla 653, Santiago, Chile}
\email{castaneda@uchile.cl,grobledo@uchile.cl}
\thanks{This research has been partially supported by MATHAMSUD program (16-MATH-04 STADE)}
\date{27 September 2016}
\begin{abstract}
We prove that, under some conditions, a linear nonautonomous difference system is Bylov's almost 
reducible to a diagonal one whose terms are contained in the Sacker and
Sell spectrum of the original system.

We also provide an example of the concept of diagonally significant 
system, recently introduced by P\"otzche. This example plays an essential role in
the demonstration of our results.
\end{abstract}

\maketitle

\section{Introduction}
Let us consider the non autonomous system of linear difference equations
\begin{equation}
\label{LinA}
x(n+1)=A(n)x(n), 
\end{equation}
where $x(n)$ is a column vector of $\mathbb{R}^{d}$ and the matrix function $n\mapsto A(n) \in \mathbb{R}^{d\times d}$ satisfies
the following properties:
\begin{itemize}
\item[\textbf{(P1)}] $A(n)$ is invertible for any $n\in \mathbb{Z}$,
\item[\textbf{(P2)}] $\sup\limits_{n\in \mathbb{Z}}||A(n)||<+\infty$ and  $\sup\limits_{n\in \mathbb{Z}}||A^{-1}(n)||<+\infty$,
\end{itemize}
where $||\cdot||$ denotes a matrix norm.

The purpose of this article is to study the contractibility or almost reducibility to a diagonal system. Namely, the $\delta$--kinematical 
similarity of (\ref{LinA}) to 
\begin{equation}
\label{LinB}
y(n+1)=U(n)y(n), 
\end{equation}
where $U(n)=U_{D}(n)\{I+U_{R}(n)\}$, where $U_{D}(n)$ is a diagonal matrix
and $U_{R}(n)$ has some smallness properties which will be explained later. 

\begin{definition}[\cite{Gohberg}]
\label{KS}
The system (\ref{LinA}) is kinematically similar (resp. $\delta$--kinematically similar with a fixed $\delta>0$) to (\ref{LinB}) if there exists an 
invertible transformation $F(n)$ (resp. $F(\delta,n)$)  verifying 
$$
\sup\limits_{n\in \mathbb{Z}}||F(n)||<+\infty \quad \textnormal{and} \quad 
\sup\limits_{n\in \mathbb{Z}}||F^{-1}(n)||<+\infty
$$ 
or respectively
$$
\sup\limits_{n\in \mathbb{Z}}||F(\delta,n)||<+\infty \quad \textnormal{and} \quad 
\sup\limits_{n\in \mathbb{Z}}||F^{-1}(\delta,n)||<+\infty,
$$ 
such that the change of coordinates $y(n)=F^{-1}(n)x(n)$ (resp.  $y(n)=F^{-1}(\delta,n)x(n)$) transforms (\ref{LinA}) into (\ref{LinB}).
\end{definition}

The concept of almost reducibility was introduced by Bylov \cite{Bylov} in the continuous context 
and the following definition is a discrete version.
\begin{definition}
\label{almost-reducibility}
The system (\ref{LinA}) is almost reducible to 
\begin{displaymath}
y(n+1)=V(n)y(n),
\end{displaymath}
if for any $\delta>0$, is $\delta$--kinematically similar to
\begin{displaymath}
y(n+1)=V(n)\{I+B(n)\}y(n),
\end{displaymath}
with 
\begin{displaymath}
||B(n)||\leq \delta \quad \textnormal{for any $n\in \mathbb{Z}$}.
\end{displaymath}
\end{definition}

In the case when $V(n)$ is a diagonal matrix it is said that (\ref{LinA}) is almost reducible to a diagonal system and
it was proved in \cite{Bylov} that any continuous linear system satisfy this property and the components
of $V(n)$ are real numbers.

The concept of almost reducibility to a diagonal system was rediscovered and improved by F. Lin in \cite{Lin},
who introduces the concept of contractibility. In this paper, we are introducing its discrete version.
\begin{definition}
\label{def-lin}
The system (\ref{LinA}) is contracted to the compact subset 
$E\subset \mathbb{R}^{+}$ if is almost reducible to a diagonal system
 \begin{displaymath}
 y(n+1)= \diag(C_1(n), \ldots, C_d(n))y(n),
 \end{displaymath}
where $C_{i}(n)\in E$ for any $n\in \mathbb{Z}$.
\end{definition}

It is worth to emphasize that while Bylov's result only says that the diagonal components are real numbers, Lin
provides explicit localization properties. This arises the following definition.
\begin{definition}
\label{contractibilidad}
The compact set $E\subset \mathbb{R}^{+}$ is said to be the contractible set of (\ref{LinA}) if
$E$ is the minimal compact set such that the system (\ref{LinA}) can be contracted.
\end{definition}

In the continuous case, the concept of contractibility has been applied in some results of topological equivalence \cite{Lin2}. 
The major contribution of Lin's article \cite{Lin} is to prove that the Sacker and Sell spectrum of $A(n)$ 
(a formal definition will be given later) is the contractible set of (\ref{LinA}). To the best of our knowledge, there are 
no results in the discrete case and the purpose of this article is to obtain conditions for the 
contractibility of (\ref{LinA}) by following some lines of Lin's work.


\subsection{Notation and terminology} The fundamental matrix of (\ref{LinA}) is defined by
\begin{equation}
\label{FM} 
X(n)=\left\{\begin{array}{rcl}
              A(n-1)\ldots A(0) & \textnormal{if} & n\geq 1\\\\
                I &\textnormal{if} & n=0\\\\
              A^{-1}(n)\ldots A^{-1}(-1) & \textnormal{if} & n\leq -1.
             \end{array}\right.
\end{equation}

The transition matrix $X(n,k)=X(n)X^{-1}(k)$ is defined as follows:
\begin{equation}
\label{TM} 
X(n,k)=\left\{\begin{array}{rcl}
              A(n-1)\ldots A(k) & \textnormal{if} & n>k\\\\
                I &\textnormal{if} & n=k\\\\
              A^{-1}(n)\ldots A^{-1}(k-1) & \textnormal{if} & n< k.
             \end{array}\right.
\end{equation}

Vector and matrix norms will be respectively denoted by $|\cdot|$ and $||\cdot||$. As usual
the infinite norm $||\cdot||_{\infty}$ is
\begin{displaymath}
||A||_{\infty}=\max\limits_{1\leq i \leq d}\left(\sum\limits_{j=1}^{d}|a_{ij}|\right).
\end{displaymath}

We will also assume the following convention for sums and products \cite[Pag.3]{Elaydi}:
\begin{displaymath}
\sum\limits_{j=n_{0}}^{n_{0}-1}a_{j}=0  \quad \textnormal{and} \quad  \prod\limits_{j=n_{0}}^{n_{0}-1}a_{j}=1. 
\end{displaymath}

\subsection{Outline} The section 2 is devoted to recall the spectral theory for linear difference systems
and state the main results. The Section 3 introduces some technical preparatory results. Furthermore, Proposition
\ref{roughness1} will provide an example of diagonal significance in the P\"otzsche's sense. The main results are proved
in the Section 4.

\section{Main result: almost reducibility to diagonal systems and spectral theory}

\subsection{Dichotomy and Sacker $\&$ Sell's Spectrum}
\begin{definition}[\cite{Papas-Schi},\cite{Papaschinopoulos-1990},\cite{Pinto},\cite{Potzsche}]
\label{ED}
The system (\ref{LinA}) has an exponential dichotomy on $\mathbb{Z}$ if
there exist numbers $K\geq 1$, $\rho\in (0,1)$ and a projector $P^{2}=P$ such that
\begin{equation}
\left\{\begin{array}{rcl}
        ||X(n)PX(k)^{-1}||\leq K\rho^{n-k} & \textnormal{if} & n\geq k,\\\\
        ||X(n)(I-P)X(k)^{-1}||\leq K\rho^{k-n} & \textnormal{if} & n \leq k.
       \end{array}\right.
\end{equation}
\end{definition}

\begin{definition}
\label{SSS}
The Sacker--Sell spectrum (also called exponential dichotomy spectrum) of (\ref{LinA}) is the set $\Sigma(A)$ of $\lambda>0$ such that
the systems
\begin{equation}
\label{pond-lamb}
x(n+1)=\lambda^{-1}A(n)x(n)
\end{equation}
have not an exponential dichotomy on $\mathbb{Z}$.
\end{definition}

\begin{remark}
\label{otra-MF}
The fundamental
matrix of (\ref{pond-lamb}) is $X_{\lambda}(n)=\lambda^{-n}X(n)$. Now,
if $\lambda \notin \Sigma(A)$, then there exist numbers $K\geq 1$, $\theta\in (0,1)$ and a 
projector $P_{\lambda}^{2}=P_{\lambda}$ such that $X_{\lambda}(n)$ satisfies
\begin{equation}
\label{A+}
\left\{\begin{array}{rcl}
        ||X(n)\lambda^{-n}P_{\lambda}X(k)^{-1}\lambda^{k}||\leq K\theta^{n-k} & \textnormal{if} & n\geq k,\\\\
        ||X(n)\lambda^{-n}(I-P_{\lambda})X(k)^{-1}\lambda^{k}||\leq K\theta^{k-n} & \textnormal{if} & n\leq k.
       \end{array}\right.
\end{equation}
\end{remark}

\begin{proposition}[Spectral Theorem]
\label{Compactness-SSS}
The Sacker--Sell spectrum $\Sigma(A)$ of (\ref{LinA}) when $||A(n)||<+\infty$ for any $n\in \mathbb{Z}$ is
either $\Sigma(A)=\emptyset$ or the union of $\ell$ closed intervals
(called spectral intervals) where $0< \ell \leq d$:
\begin{equation}
\label{espectro}
\Sigma(A)=[a_{1},b_{1}]
                   \cup [a_{2},b_{2}] \cup \cdots \cup [a_{\ell-1},b_{\ell-1}]\cup 
                  [a_{\ell},b_{\ell}],              
\end{equation}
where $0<a_{1}\leq b_{1}<a_{2}\leq b_{2}<\cdots <a_{\ell}\leq b_{\ell}$.
\end{proposition}

\begin{remark}
\label{int-nospec}
If $\lambda \notin \Sigma(A)$, it follow from definition that (\ref{pond-lamb}) has an exponential dichotomy on $\mathbb{Z}$ with projector 
$P_{\lambda}$. Nevertheless, it is interesting to note that:
\begin{itemize}
 \item[a)] $\rank(P_{\lambda})$ is constant for any $\lambda \in (b_{i-1},a_{i})$ ($i=2,\ldots,\ell$),
 \item[b)] If $\lambda_{i} \in  (b_{i-1},a_{i})$ and $\lambda_{i+1}\in (b_{i},a_{i+1})$, then
 $\rank(P_{\lambda_{i}})<\rank(P_{\lambda_{i+1}})$,
 \item[c)] $\rank(P_{\lambda})=0$ for
any $\lambda \in (0,a_{1})$ and $\rank(P_{\lambda})=d$ for
any $\lambda \in (b_{\ell},\infty)$.
\end{itemize}
\end{remark}

It is interesting to point out that Siegmund and Aulbach \cite{Aulbach2000},\cite{Aulbach2001} developed an
spectral theory directly from (\ref{LinA}) avoiding the technicalities from linear skew--product flows, which are used in the original
work of Sacker and Sell \cite{SS}. This approach is widely used in the current research and will simplify our work.

\begin{definition}
The system (\ref{LinA}) has the full spectrum condition if
\begin{displaymath}
\Sigma(A)= [a_{1},b_{1}] \cup [a_{2},b_{2}] \cup \cdots \cup [a_{d},b_{d}].
\end{displaymath}
\end{definition}

\subsection{Main results}
In order to contextualize our main results, let us consider the scalar difference equation studied in \cite{Aulbach2000}:
\begin{equation}
\label{baby1}
x_{n+1}=a_{n}x_{n}, \quad \textnormal{with} \quad a_{n}=\left\{\begin{array}{rcl}
                                                              a & \textnormal{if} & n\leq -1\\
                                                              b & \textnormal{if} & n\geq 0
                                                             \end{array}\right. \quad \textnormal{and} \quad 0<a\leq b.
\end{equation}

We claim that (\ref{baby1}) is contracted to any interval $[c,d]$ with $0<c\leq a\leq b \leq d$. Indeed, given a fixed $\delta>0$, we consider
\begin{displaymath}
F(n,\delta)=\left\{\begin{array}{rcl}
              \prod\limits_{j=0}^{n-1}(1+\frac{\delta \cos^{2}(j)}{j^{2}+1})^{-1}F_{0} &\textnormal{if} & n\geq 1\\\\ 
              \prod\limits_{j=n}^{-1}(1+\frac{\delta \cos^{2}(j)}{j^{2}+1})F_{0} &\textnormal{if} & n\leq -1,
              \end{array}\right.
\end{displaymath}
where $F_{0}\neq 0$ and we can verify that (\ref{baby1}) is $\delta$--kinematically similar to
\begin{displaymath}
y_{n+1}=a_{n}\{1+b_{n}\}y_{n} \quad \textnormal{with} \quad b_{n}=\frac{\delta \cos^{2}(n)}{1+n^{2}}.
\end{displaymath}

The claim follows since $a_{n}\in [a,b] \subseteq [c,d]$ and $|b_{n}|\leq \delta$. Morevoer, it is shown 
in \cite{Aulbach2000} that the linear difference equation (\ref{baby1}) has $\Sigma(A)=[a,b]$ and we
verified that (\ref{baby1}) can be contracted to any closed interval containing $[a,b]$. In addition, $[a,b]$ is 
the minimal interval where the system can be contracted. This fact illustrates our 
main result.

\begin{theorem}
\label{Main}
If the spectrum $\Sigma(A)$ is a bounded set of $\mathbb{R}^{+}$, then is a contractible set of (\ref{LinA}). 
\end{theorem}

\begin{theorem}
\label{Main2}
If (\ref{LinA}) has the full spectrum condition, then is diagonalizable. 
\end{theorem}


\begin{remark}
Note that
\begin{itemize}
\item[a)] A careful reading of the proof of Theorem \ref{Main} will show that the coefficient $\delta>0$ of
Definition \ref{def-lin} has an upper bound dependent of $\Sigma(A)$.
\item[b)] In spite that Theorem \ref{Main2} is not new, we point out the remarkable simplicity of our proof which follows the lines
of the proof of Theorem \ref{Main}.
\end{itemize}
\end{remark}

\section{Preparatory results}
The kinematically similar between (\ref{LinA}) and (\ref{LinB}) will be denoted by $A \simeq U$. Let us recall
that kinematical similarity is an equivalence relation and have several properties described in the following
lemmatas:
\begin{lemma}
\label{KS-escalar} 
If $A \simeq B$, then $\mu A \simeq \mu B$ for any $\mu \in \mathbb{R}^{+}$.
\end{lemma}

\begin{lemma}
\label{preservacion1}
If $A\simeq B$ and (\ref{LinA}) has an exponential dichotomy on $\mathbb{Z}$, then 
(\ref{LinB}) has also an exponential dichotomy with the same
projector and constant $\rho$.
\end{lemma}

The proof of these results is a straightforward exercise and can be done similarly as in the continuous 
case.

\begin{lemma}
\label{KS-Spec}
If $A\simeq B$, then $\Sigma(A)=\Sigma(B)$. 
\end{lemma}

\begin{proof}
Firstly, if $A\simeq B$, Lemma \ref{KS-escalar} implies  $\lambda^{-1}A\simeq \lambda^{-1}B$. Secondly
if $\lambda \notin \Sigma(A)$,  Lemma \ref{preservacion1} implies that  $\lambda \notin \Sigma(B)$,
which is equivalent to $\Sigma(B)\subseteq \Sigma(A)$. The inverse contention can be proved analogously.
\end{proof}

The following results give us useful properties 
of the spectrum $\Sigma(A)$:

\begin{lemma}
\label{intervalos-ponderados}
If $||A(n)||<\infty$ for any $n\in \mathbb{Z}$ and $\lambda \notin \Sigma(A)$, then
\begin{equation}
\label{SS}
\Sigma(\lambda^{-1}A)=\Big[\frac{a_{1}}{\lambda},\frac{b_{1}}{\lambda}\Big] \cup
                   \Big[\frac{a_{2}}{\lambda},\frac{b_{2}}{\lambda}\Big] \cup \cdots \cup \Big[\frac{a_{\ell-1}}{\lambda},\frac{b_{\ell-1}}{\lambda}\Big]\cup 
                  \Big[\frac{a_{\ell}}{\lambda},\frac{b_{\ell}}{\lambda}\Big]
\end{equation}
\end{lemma}

\begin{proof}
By definition of spectrum, we know that
\begin{displaymath}
\Sigma(\lambda^{-1}A)=\Big\{\mu\in \mathbb{R}^{+} \colon x(n+1)=\frac{1}{\mu\lambda}A(n)x(n) \quad  \textnormal{has not an exponential dichotomy on $\mathbb{Z}$}\Big\}
\end{displaymath}

Notice that 
\begin{displaymath}
\begin{array}{rcl}
\mu \in \Sigma(\lambda^{-1}A) & \Leftrightarrow & \frac{1}{\mu\lambda} \in \Sigma(A)\\\\
                              & \Leftrightarrow & a_{j}\leq \mu\lambda \leq b_{j} \quad \textnormal{for some $j=1,\ldots,\ell$}\\\\
                              & \Leftrightarrow & \frac{a_{j}}{\lambda} \leq \mu \leq \frac{b_{j}}{\lambda}
\end{array}
\end{displaymath}
and the result follows.
\end{proof}

\begin{proposition}
\label{invariance1}
If $\Sigma(A) \subseteq [a,b],$ and $\lambda > b$ (resp. or $\lambda < a$), the system
$$
x(n+1)=\lambda^{-1}A(n)x(n)
$$
has an exponential dichotomy on $\mathbb{Z}$ with projector $P=I$ (resp. with projector $P=0$) and 
constants $K=1$ and $\theta\in (0,1)$.
\end{proposition}

\begin{proof}
We will consider the case $\lambda>b$, the other one can be proved by the reader in a similar way. By using the 
discrete Gronwall's inequality \cite[Ch.4]{Elaydi}, it is easy to deduce
\begin{displaymath}
||X(n)X^{-1}(k)||< (1+L)^{n-k} \quad \textnormal{where} \quad n\geq k \quad \textnormal{and}\quad L=\sup\limits_{n\in \mathbb{Z}}||A_{n}-I||.
\end{displaymath}

 Let $h=\max\{L+2,\lambda\}$ and note that $(1+L)/h \in (0,1)$ and
\begin{equation}
\label{Linh}
||X(n)\big(\frac{1}{h}\big)^{n}X^{-1}(k)\big(\frac{1}{h}\big)^{-k}||< \Big(\frac{1+L}{h}\Big)^{n-k} \quad n\geq k, 
\end{equation}
which implies that the system
$$
x(n+1)=h^{-1}A(n)x(n)
$$
has an exponential dichotomy on $\mathbb{Z}$ with projector $P_{h}=I$. By using Remark \ref{int-nospec}, we know that (\ref{pond-lamb})
also has an exponential dichotomy on $\mathbb{Z}$ with $P_{\lambda}=P_{h}=I$ for any $\lambda>b$.
\end{proof}

 \begin{remark}
  The result above 
  gives explicit constants for the exponential dichotomy as $K=1$ and $\theta=(1+L)/h\in (0,1)$. This fact will be useful
  in some future steps.
 \end{remark}

 The following result has been proved by Siegmund in \cite{Siegmund} in a more general case with conditions less restrictive than 
 \textbf{(P1)--(P2)}. Moreover, the exponential dichotomy considers
 variable projectors. Nevertheless, we provide a proof in order to make the article the most self--contained possible.

\begin{proposition}
\label{bloques}
If (\ref{LinA}) verifies properties \textnormal{\textbf{(P1)--(P2)}} and its spectrum
is of type (\ref{espectro}), then there exist $\ell$ matrices $B_i(n)$
where $\Sigma(B_i) = [a_i, b_i]$ with $i = 1, \ldots, \ell,$ such that (\ref{LinA}) is kinematically similar to the system
\begin{equation}
\label{diag2}
x(n+1) = \diag(B_1(n), \ldots, B_\ell(n))x(n).
\end{equation}
\end{proposition}

\begin{proof}
Let us choose $\lambda_{\ell} \in (b_{\ell-1},a_{\ell})$ then the linear system
\begin{equation}
\label{ponderado}
x(n+1)=\lambda_{\ell}^{-1}A(n)x(n)
\end{equation}
has an exponential dichotomy with projector $P_{\lambda_{\ell}}$ with $\rank(P_{\lambda_{\ell}})=m_{\ell}<d$
(see Remark \ref{int-nospec} for details).

By using Lemma 1 from \cite{Papaschinopoulos}, we know that (\ref{ponderado}) is kinematically similar
to 
\begin{equation}
\label{ponderado-ks}
y(n+1) = \left (\begin{array}{ccc}
A_1(n) & & 0\\
0 & & A_2(n)
\end{array}
\right ) y(n),
\end{equation}
where $A_{1}\in \mathbb{R}^{m_{\ell}\times m_{\ell}}$ and  $A_{2}\in \mathbb{R}^{(d-m_{\ell})\times (d-m_{\ell})}$. In addition
(see \emph{e.g.}, \cite[p.281]{Kurzweil}), the subsystem
\begin{equation}
\label{ecu1}
y_{1}(n+1)=A_{1}(n)y_{1}(n) 
\end{equation}
has an exponential dichotomy on $\mathbb{Z}$ with the identity as a projector and 
\begin{equation}
\label{ecu2}
y_{2}(n+1)=A_{2}(n)y_{2}(n) 
\end{equation}
has an exponential dichotomy on $\mathbb{Z}$ with the null projector.

As a consequence of Lemmas \ref{KS-Spec} and \ref{intervalos-ponderados}, we know that
\begin{displaymath}
\Sigma(\lambda_{\ell}^{-1}A)=\bigcup\limits_{i=1}^{\ell}\Big[\frac{a_{i}}{\lambda_{\ell}},\frac{b_{i}}{\lambda_{\ell}}\Big]=\Sigma(A_{1})\cup \Sigma(A_{2}),  
\end{displaymath}
where 
\begin{displaymath}
\Sigma(A_{1})=\bigcup\limits_{i=1}^{j-1}\Big[\frac{a_{i}}{\lambda_{\ell}},\frac{b_{i}}{\lambda_{\ell}}\Big] \quad \textnormal{and} \quad
\Sigma(A_{2})=\bigcup\limits_{i=j}^{\ell}\Big[\frac{a_{i}}{\lambda_{\ell}},\frac{b_{i}}{\lambda_{\ell}}\Big]
\end{displaymath}
for some $j\in \{1,\ldots,\ell\}$.

Now, we will verify that $j=\ell$ and $\Sigma(A_{2})=[\frac{a_{\ell}}{\lambda_{\ell}},\frac{b_{\ell}}{\lambda_{\ell}}]$. Indeed, let us recall that
\begin{displaymath}
\Sigma(A_{2})=\Big\{\mu \in \mathbb{R}^{+}\colon u(n+1)=\mu^{-1}A_{2}(n)u(n) \quad \textnormal{has not an exponential dichotomy on $\mathbb{Z}$}\Big\} 
\end{displaymath}
and note that $1\notin \Sigma(A_{2})$ since (\ref{ecu2}) has an exponential dichotomy on $\mathbb{Z}$ with the null projector. In addition, let us recall that $1 \in (\frac{b_{\ell-1}}{\lambda_{\ell}},\frac{a_{\ell}}{\lambda_{\ell}})$. Now, 
if $j<\ell$ and $\mu \in (\frac{b_{\ell-2}}{\lambda_{\ell}},\frac{a_{\ell-1}}{\lambda_{\ell}})$,
the system
\begin{displaymath}
y_{2}(n+1)=\mu^{-1}A_{2}(n)y_{2}(n) 
\end{displaymath}
will have an exponential dichotomy with a projector $Q$ and statements (a)--(c) from Remark \ref{int-nospec} says 
that $\rank(Q)$ must be lower than zero, obtaining a contradiction.

As (\ref{ponderado}) and (\ref{ponderado-ks}) are kinematically similar, Lemma \ref{KS-escalar} implies that
(\ref{LinA}) is kinematically similar to
\begin{displaymath}
\begin{array}{rcl}
y(n+1) &=& \left (\begin{array}{ccc}
\lambda_{\ell} A_1(n) & & 0\\
0 & & \lambda_{\ell} A_2(n)
\end{array}
\right ) y(n)\\\\
&=& \left (\begin{array}{ccc}
B_{0}(n) & & 0\\
0 & & B_\ell(n)
\end{array}
\right ) y(n)
\end{array}
\end{displaymath}
and note that Lemma \ref{intervalos-ponderados} implies
\begin{displaymath}
\Sigma(B_{0})=\bigcup\limits_{i=1}^{\ell-1}\big[a_{i},b_{i}\big] \quad \textnormal{and} \quad
\Sigma(B_{\ell})=\big[a_{\ell},b_{\ell}\big].
\end{displaymath}

Now, let us consider the system
\begin{displaymath}
z(n+1)=B_{0}(n)z(n) 
\end{displaymath}
and take $\lambda_{\ell-1} \in (b_{\ell-2},a_{\ell-1})$. Then, the system
\begin{equation}
\label{eq-final}
w(n+1)=\frac{1}{\lambda_{\ell-1}}B_{0}(n)w(n) 
\end{equation}
has exponential dichotomy with projector $P_{\lambda_{\ell-1}}$ with $\rank(P_{\lambda_{\ell-1}})=m_{\ell-1}<m_{\ell}$.

The system (\ref{eq-final}) can be studied similarly as (\ref{ponderado}) in the paragraphs above and the proof can be achieved recursively.
\end{proof}

An important matter of spectral theory for differential \cite{Palmer2015} and difference \cite{Potzsche} nonautonomous systems 
is to determine sufficient conditions ensuring \emph{diagonal significance} in the P\"otzche's sense \cite{Potzsche}, namely that 
the spectrum $\Sigma(C)$ of an upper triangular system $u(n+1)=C(n)u(n)$ coincides with 
the union of the diagonal spectra $\Sigma(c_{ii})$. The following result provides an example of diagonal significance 
and will play a fundamental role in the proof of our main result. 

\begin{proposition}
\label{roughness1}
Let $C(n)=\{c_{ij}(n)\}$ be a bounded and upper triangular $d\times d$--matrix function such that $\Sigma(C)=[a,b]$. Then
\begin{displaymath}
\Sigma(C)=\bigcup\limits_{i=1}^{d}\Sigma(c_{ii}).
\end{displaymath}
\end{proposition}

\begin{proof}
Firstly, we will prove that 
$\Sigma(C)\subseteq \bigcup\limits_{i=1}^{d}\Sigma(c_{ii})$. Let $\lambda \notin \bigcup\limits_{i=1}^{d}\Sigma(c_{ii})$, then 
the diagonal system
\begin{displaymath}
x(n+1)=\lambda^{-1}\diag\big(c_{11}(n),\ldots,c_{dd}(n)\big)x(n) 
\end{displaymath}
has an exponential dichotomy. Now, let us consider the upper triangular system
\begin{equation}
\label{SPD}
x(n+1)=\lambda^{-1} C(n)x(n),
\end{equation}
where 
\begin{equation}
\label{CotaL}
C^{+}=\sup\limits_{n\in \mathbb{Z}}\|C(n)\|.
\end{equation}

Let us make the change of variables (usually known as $\beta$--transformation):
\begin{displaymath}
y(n)=D_{\beta}x(n)=\diag(1,\beta,\ldots,\beta^{d-1})^{-1}x(n) \quad \textnormal{where} \quad 0<\beta<\frac{\delta}{\delta+ C^{+}}
\end{displaymath}
and notice that
\begin{displaymath}
\begin{array}{rcl}
y(n+1)&=&D_{\beta}^{-1}x(n+1)\\\\
 &=&D_{\beta}^{-1}
\left(\begin{array}{cccc}
\frac{1}{\lambda}c_{11}(n) & c_{12}(n) & \cdots & c_{1d}(n)\\
                  & \ddots    &         & \vdots  \\
                  &           &  \ddots  & \vdots \\
                  &           &          & \frac{1}{\lambda}c_{dd}(n)
                  \end{array}\right)
D_{\beta}\,y(n)\\\\
&=&\left[
     \frac{1}{\lambda}\left(\begin{array}{cccc}
c_{11}(n) &  & & \\
                  & \ddots    &         &   \\
                  &           &  \ddots  &  \\
                  &           &          & c_{dd}(n)
                  \end{array}\right)+
\left(\begin{array}{cccc}
0 & \beta c_{12}(n) & \cdots & \beta^{d-1}c_{1d}(n)\\
  & \ddots         &        &  \vdots            \\
  & \ddots         & \ddots & \beta c_{d-1\, d}(n)   \\
  &                &        & 0
  \end{array}\right)\right]y(n).
\end{array}
\end{displaymath}

 By using $\|\cdot\|_{\infty}$ norm combined with $\beta<1$,  we can see that
 \begin{displaymath}
\left\| \left(\begin{array}{cccc}
0 & \beta c_{12}(n) & \cdots & \beta^{d-1}c_{1d}(n)\\
  & \ddots         &        &  \vdots            \\
  & \ddots         & \ddots & \beta c_{d-1\, d}(n)   \\
  &                &        & 0
  \end{array}\right)\right\|_{\infty}\leq \delta.
 \end{displaymath}
Due to roughness results for difference equations \cite[Corollary 3]{Aulbach1994}
(see also  \cite[p.276]{Palmer},\cite[Proposition 1]{Papaschinopoulos-1988}), we know that if $\delta$ is small enough, the system
 \begin{displaymath}
y(n+1)=\left[
   \frac{1}{\lambda}  \left(\begin{array}{cccc}
c_{11}(n) &  & & \\
                  & \ddots    &         &   \\
                  &           &  \ddots  &  \\
                  &           &          & c_{dd}(n)
                  \end{array}\right)+
\left(\begin{array}{cccc}
0 & \beta c_{12}(n) & \cdots & \beta^{d-1}c_{1d}(n)\\
  & \ddots         &        &  \vdots            \\
  & \ddots         & \ddots & \beta c_{d-1\, d}(n)   \\
  &                &        & 0
  \end{array}\right)\right]y(n)
 \end{displaymath}
 has an exponential dichotomy. By construction, this system is kinematically similar to (\ref{SPD}) 
 and Lemma \ref{preservacion1} says that (\ref{SPD}) has also an exponential dichotomy, which implies that 
 $\lambda \notin \Sigma(C)$ and $\Sigma(C)\subseteq \bigcup\limits_{i=1}^{d}\Sigma(C_{ii})$ follows.

 Secondly, we will prove that $\bigcup\limits_{i=1}^{d}\Sigma(c_{ii})\subseteq \Sigma(C)$.  
 Let $\lambda \notin \Sigma(C)=[a,b]$ such that $\lambda>b$. By Proposition  \ref{invariance1},
  we have that the system (\ref{SPD}) has exponential dichotomy with projection $P=I$. That is, the fundamental matrix 
  of (\ref{SPD}), namely $X_{\lambda}$,
  satisfies
  \begin{displaymath}
  \left\|X_{\lambda}(n)X_{\lambda}^{-1}(k)\right\|_{\infty}\leq \rho^{n-k} \quad \textnormal{with $n\geq k$}.
  \end{displaymath}

 Let us recall that $X_{\lambda}(n)=X(n)\lambda^{-n}$, where $X(n)$ is the fundamental matrix of the  
 system $x(n+1)=C(n)x(n)$.  Now, as $C(n)$ is upper triangular, we can see with the help of (\ref{TM}) that 
  \begin{displaymath}
  X_{\lambda}(n)X_{\lambda}^{-1}(k)=\left\{
\begin{array}{ccl} 
 \lambda^{k-n}C(n-1)C(n-2)\cdots C(k)  & \textnormal{if} & n>k \\
 I                        & \textnormal{if} & n=k,    
\end{array}\right.
\end{displaymath}
which implies that
  
   \begin{displaymath}
   \lambda^{k-n}\prod\limits_{j=k}^{n-1}|c_{ii}(j)| \leq  \left\|X_{\lambda}(n)X_{\lambda}^{-1}(k)\right\|_{\infty}\leq \rho^{n-k} 
   \quad \textnormal{for any} \quad i=1,\ldots,d,
   \end{displaymath}
and we conclude that each scalar difference equation $x_{i}(n+1)=\lambda^{-1}c_{ii}(n)x_{i}(n)$  ($i=1,\ldots,d$) has 
an exponential dichotomy with projection $1$, which implies that $\lambda \notin \bigcup\limits_{i=1}^{n}\Sigma(c_{ii})$.

   The case $\lambda <a$ can be proved analogously, thus  $\bigcup\limits_{i=1}^{n}\Sigma(c_{ii})\subseteq \Sigma(C)$ and the 
   Proposition follows.
\end{proof}

\begin{proposition}
\label{minimal}
If the linear system (\ref{LinA}) satisfies \textnormal{\textbf{(P1)--(P2)}} and can be contracted to a compact set $E\subset (0,+\infty)$, 
then $\Sigma(A)\subseteq E$.
\end{proposition}

\begin{proof}
Let us choose $\lambda \notin E$ and notice that the compactness of $E$ allow to define
$\alpha=\inf_{x\in E}|\lambda-x|>0$.
By using Definition \ref{contractibilidad}, we have that (\ref{LinA}) is kinematically similar to
\begin{displaymath}
y(n+1)=\diag(C_{1}(n),\ldots,C_{d}(n))\{I+B(n)\}y(n),
\end{displaymath}
where  $C_{i}(n)\in E$ for any $n\in \mathbb{Z}$ and $\sup\limits_{n\in \mathbb{Z}}\|B(n)\|<\delta/||C||$. Now, by Lemma \ref{KS-escalar} we know that 
(\ref{pond-lamb}) is $\delta$--kinematically similar to
\begin{equation}
\label{lambda2}
y(n+1)=\frac{1}{\lambda}\diag(C_{1}(n),\ldots,C_{d}(n))\{I+B(n)\}y(n).
\end{equation}

Since $C_{i}(n)\in E$ for any $n\in \mathbb{Z}$ and $i=1,\ldots,d$, without loss of generality, we can assume that
\begin{displaymath}
\begin{array}{rcl}
C_{i}(n)<\lambda & \textnormal{if}& i=1,\ldots,m\\
C_{i}(n)>\lambda & \textnormal{if} & i=m+1,\ldots,d.
\end{array}
\end{displaymath}

By definition of $\alpha$, we have that
\begin{displaymath}
\begin{array}{rcl}
C_{i}(n)<\lambda-\alpha & \textnormal{if}& i=1,\ldots,m\\
C_{i}(n)>\lambda+\alpha & \textnormal{if} & i=m+1,\ldots,d.
\end{array}
\end{displaymath}

We can verify that the system
\begin{displaymath}
y(n+1)=\frac{1}{\lambda}\diag(C_{1}(n),\ldots,C_{d}(n))y(n),
\end{displaymath}
has an exponential dichotomy on $\mathbb{Z}$ since
\begin{displaymath}
\begin{array}{rcl}
|(\frac{1}{\lambda})^{n-k}\prod\limits_{j=k}^{n-1}C_{i}(j)|\leq \max\{\frac{\lambda-\alpha}{\lambda},\frac{\lambda}{\lambda+\alpha}\}^{n-k} &\textnormal{if}& n\geq k, \quad i=1,\ldots, m\\\\
|(\frac{1}{\lambda})^{n-k}\prod\limits_{j=n}^{k-1}C_{i}^{-1}(j)|\leq \max\{\frac{\lambda-\alpha}{\lambda},\frac{\lambda}{\lambda+\alpha}\}^{k-n} &\textnormal{if}&  n\leq k, \quad (i=m+1,\ldots,d).
\end{array}
\end{displaymath}

By using roughness results, we have that (\ref{lambda2}) has an exponential dichotomy on $\mathbb{Z}$. Now, due to
kinematical similarity and Lemma \ref{preservacion1}, the system (\ref{pond-lamb}) has exponential dichotomy. In 
consequence, $\lambda\notin \Sigma(A)$ and 
the Proposition follows.
\end{proof}

\section{Proof of Main results}

\subsection{Proof of Theorem \ref{Main}} The proof will be made in several steps:

\noindent\emph{Step 1): (\ref{LinA}) is kinematically similar to an upper triangular system:} By Proposition \ref{Compactness-SSS} and 
hypothesis, there exists a positive integer $\ell \leq d$ such that:
$$
\Sigma(A)=\bigcup\limits_{i=1}^{\ell}[a_{i},b_{i}], \quad \textnormal{with $0<a_{1}\leq b_{1}<a_{2}\leq b_{2}<\cdots<a_{\ell}\leq b_{\ell}<+\infty$}.
$$

By Proposition \ref{bloques}, we know that (\ref{LinA})
is kinematically similar to (\ref{diag2}), where $B_{i}(n)$ are bounded 
$m_{i}\times m_{i}$ matrix functions and $\Sigma(B_{i})=[a_{i},b_{i}]$ ($i=1,\ldots,\ell$). 

By using the method of QR factorization, we know that, for each $i\in\{1,\ldots,\ell\}$, the
systems
\begin{equation}
\label{6B}
x_{i}(n+1)=B_{i}(n)x_{i}(n)
\end{equation}
are kinematically similar to the $m_{i}\times m_{i}$ upper triangular systems 
\begin{equation}
\label{6D}
y_{i}(n+1)=C_{i}(n)y_{i}(n),
\end{equation}
where $\Sigma(C_{i})=[a_{i},b_{i}]$.

\bigskip

\noindent\emph{Step 2) Exponential dichotomy of scalar difference systems:} From now on, the diagonal terms of the upper triangular matrices $C_{i}$
described in (\ref{6D}) will be denoted by $\{c_{i_{rr}}\}_{r=1}^{m_{i}}$, where $i$ is a fixed 
element of $\{1,\ldots,\ell\}$. Now, by Proposition \ref{roughness1}, we know that the upper triangular system 
(\ref{6D}) has the property of diagonal significance,
that is
$$
\bigcup\limits_{r=1}^{m_{i}}\Sigma(c_{i_{rr}})=[a_{i},b_{i}].
$$

By Proposition \ref{invariance1}, for any $\delta\in (0,2a_{1})$, the scalar difference equation
\begin{equation}
\label{primera}
u(n+1)=\frac{c_{i_{rr}}(n)}{\big(a_{i}-\frac{\delta}{2}\big)}u(n)
\end{equation}
has exponential dichotomy with projector $P=0$ and
\begin{equation}
\label{segunda}
s(n+1)=\frac{c_{i_{rr}}(n)}{\big(b_{i}+\frac{\delta}{2}\big)}s(n)
\end{equation}
has exponential dichotomy with constant $K=1$ and projector $P=1$. In consequence, there exists $\theta \in (0,1)$ such that
\begin{equation}
\label{CotaKS}
\left\{\begin{array}{rcl}
|U(n)U^{-1}(k)| \leq \theta^{k-n} &\textnormal{if}& k\geq n\\\\
|S(n)S^{-1}(k)| \leq \theta^{n-k} &\textnormal{if} & n\geq k.
\end{array}\right.
\end{equation}

Where
\begin{displaymath}
U(n)=\left\{\begin{array}{ccl}
\displaystyle \prod\limits_{j=0}^{n-1}\bigg(\frac{c_{i_{rr}}(j)}{a_{i}-\frac{\delta}{2}}\bigg)&\textnormal{if}& n\geq 1\\\\
1 & \textnormal{if}& n=0\\\\
\displaystyle \prod\limits_{j=n}^{-1}\bigg(\frac{c_{i_{rr}}(j)}{a_{i}-\frac{\delta}{2}}\bigg)^{-1} &\textnormal{if}& n< 0
\end{array}\right. \quad \textnormal{and} \quad
S(n)=\left\{\begin{array}{ccl}
\displaystyle \prod\limits_{j=0}^{n-1}\bigg(\frac{c_{i_{rr}}(j)}{b_{i}+\frac{\delta}{2}}\bigg)&\textnormal{if}& n\geq 1\\\\
1 & \textnormal{if}& n=0\\\\
\displaystyle \prod\limits_{j=n}^{-1}\bigg(\frac{c_{i_{rr}}(j)}{b_{i}+\frac{\delta}{2}}\bigg)^{-1} &\textnormal{if}& n< 0
\end{array}\right.
\end{displaymath}
are the fundamental matrices of (\ref{primera}) and (\ref{segunda}) respectively.

\bigskip

\noindent\emph{Step 3) A technical result:} 
\begin{lemma}
\label{technical-result}
For any fixed $i\in \{1,\ldots,\ell\}$, there exist two sequences $h_{i}(n)$ and $\Delta_{i}(n)$ such that
\begin{equation}
\label{Cota21}
a_{i}\leq h_{i}(n)\leq b_{i} \quad \textnormal{and} \quad |\Delta_{i}(n)|\leq  \frac{\delta}{2} \quad \textnormal{for any $n\in \mathbb{Z}$}
\end{equation}
and there exists $0<M_{2}<1<M_{1}$ verifying
\begin{equation}
\label{Cota22}
M_{2}\leq \prod\limits_{j=0}^{n-1}\Big|\frac{c_{i_{rr}}(j)}{h_{i}(j)+\Delta_{i}(j)}\Big| \leq M_{1}, \quad \textnormal{if $n\geq 1$} 
\end{equation}
and
\begin{equation}
\label{Cota222}
M_{2}\leq \prod\limits_{j=n}^{-1}\Big|\Big(\frac{c_{i_{rr}}(j)}{h_{i}(j)+\Delta_{i}(j)}\Big)^{-1}\Big| \leq M_{1}, \quad \textnormal{if $n\leq -1$} 
\end{equation}
for any $r\in \{1,\ldots,m_{i}\}$.
\end{lemma}
\begin{proof}
See Appendix A.
\end{proof}

As a consequence of this result, we construct the $m_{i}\times m_{i}$ matrix:
\begin{displaymath}
L_{i}(n)=\diag\Big(\mu_{1}(n),\ldots,\mu_{m_{i}}(n)\Big), 
\end{displaymath}
where for any $r\in \{1,\ldots,m_{i}\}$, it follows that
\begin{displaymath}
\mu_{r}(n)=\left\{\begin{array}{ccl}
                 \displaystyle \prod\limits_{j=0}^{n-1}\frac{c_{i_{rr}}(j)}{h_{i}(j)+\Delta_{i}(j)} & \textnormal{if} & n\geq 0 \\\\
                 \displaystyle  \prod\limits_{j=n}^{-1}\Big(\frac{c_{i_{rr}}(j)}{h_{i}(j)+\Delta_{i}(j)}\Big)^{-1} & \textnormal{if} & n\leq -1
                 \end{array}\right.\quad \textnormal{for any $n\in \mathbb{Z}$}.
\end{displaymath}

By Lemma \ref{technical-result}, we have that 
\begin{equation}
\label{NormaL}
||L_{i}(n)||_{\infty}\leq M_{2} \hspace{0.2cm} \textnormal{and} \hspace{0.2cm} ||L_{i}^{-1}(n)||_{\infty}\leq M_{1}\hspace{0.2cm} \textnormal{for any $n\in \mathbb{Z}$}.
\end{equation}

\bigskip

\noindent\emph{Step 4) The systems (\ref{6B}) can be contracted to $[a_{i},b_{i}]$ for any $i=1,\ldots,\ell$:}  The system (\ref{6D})
is kinematically similar to 
\begin{equation}
\label{9}
z_{i}(n+1)=\Lambda_{i}(n)z_{i}(n),
\end{equation}
with $y_{i}(n)=L_{i}(n)z_{i}(n)$, where $\Lambda_{i}(n)=L_{i}^{-1}(n+1)C_{i}(n)L_{i}(n)$ is a $m_{i}\times m_{i}$ matrix 
whose $rs$--coefficient is defined by
\begin{displaymath}
\{\Lambda_{i}(n)\}_{rs}=\left\{ 
\begin{array}{ccl}
\displaystyle h_{i}(n)+\Delta_{i}(n) & \textnormal{if}&  r=s\\\\
\displaystyle \frac{\mu_{s}(n)}{\mu_{r}(n+1)}c_{rs}(n) &\textnormal{if} & 1\leq r<s\leq m_{i}\\\\
0 &\textnormal{if} & 1\leq s<r\leq m_{i}.
\end{array}\right.
\end{displaymath}

Let us define the $\beta$--transformation
\begin{displaymath}
z_{i}(n)=\diag\big(1,\beta,\ldots,\beta^{m_{i}-1}\big)w_{i}(n), 
\end{displaymath}
where
\begin{equation}
\label{cota-beta}
0<\beta<\frac{\delta}{\delta +\frac{2 C^{+}M_{1}}{M_{2}}}, 
\end{equation}
and $C^{+}$ is defined as in (\ref{CotaL}). 

Now, we can see that (\ref{9}) and (\ref{6D}) are $\delta$--kinematically similar to
\begin{displaymath}
w_{i}(n+1)=\Gamma_{i}(n)w_{i}(n), 
\end{displaymath}
where the $rs$--coefficient of $\Gamma_{i}(n)$ is
\begin{displaymath}
\{\Gamma_{i}(n)\}_{rs}=\left\{ 
\begin{array}{ccl}
\big\{\Lambda_{i}(n)\big\}_{rs} & \textnormal{if}&  r=s\\\\
\beta^{s-r}\big\{\Lambda_{i}(n)\big\}_{rs} \ &\textnormal{if} & 1\leq r<s\leq m_{i}\\\\
0 &\textnormal{if} & 1\leq s<r\leq m_{i}.
\end{array}\right.
\end{displaymath}

Let us observe that $\Gamma_{i}(n)$ can be written as follows:
\begin{displaymath}
\Gamma_{i}(n)=h_{i}(n)I+R_{i}(n), 
\end{displaymath}
where the $rs$--coefficient of $R_{i}(n)$ is defined by
\begin{displaymath}
\{R_{i}(n)\}_{rs}=\left\{ 
\begin{array}{ccl}
\displaystyle \Delta_{i}(n) & \textnormal{if}&  r=s\\\\
\displaystyle \beta^{s-r}\frac{\mu_{s}(n)}{\mu_{r}(n+1)}c_{rs}(n) &\textnormal{if} & 1\leq r<s\leq m_{i}\\\\
0 &\textnormal{if} & 1\leq s<r\leq m_{i}.
\end{array}\right.
\end{displaymath}
By (\ref{Cota21}), (\ref{Cota22}) and (\ref{Cota222}), we can verify that 
\begin{displaymath}
\begin{array}{rcl}
||R_{i}(n)|| & \displaystyle \leq & \frac{\delta}{2}+\frac{M_{1}C^{+}}{M_{2}}(\beta+\beta^{2}+\ldots+\beta^{m_{i}})\\\\
             & \displaystyle \leq & \frac{\delta}{2}+ \frac{M_{1}C^{+}}{M_{2}}\frac{\beta}{1-\beta}
\end{array}
\end{displaymath}
and by using (\ref{cota-beta}) it follows that $||R_{i}(n)||_{\infty}<\delta$ for any $n\in \mathbb{Z}$.

\bigskip

Summarizing, for any $i=1,\ldots,\ell$, the system (\ref{6B}) is $\delta$--kinematically similar
to 
\begin{displaymath}
w_{i}(n+1)=\diag(h_{i}(n),\ldots,h_{i}(n))\big\{I+h_{i}^{-1}(n)R_{i}(n)\big\}w_{i}(n),
\end{displaymath}
where $h_{i}(n)\subseteq [a_{i},b_{i}]=\Sigma(B_{i})$ and $||\frac{1}{h_{i}(n)}R_{i}(n)||_{\infty}<\frac{\delta}{|b_{i}|}$
and the contractibility of (\ref{6B}) follows.

\bigskip

\noindent\emph{Step 5) The system (\ref{LinA}) can be contracted to $\Sigma(A)$:}
By using the previous result, we can see that (\ref{LinA}) is $\tilde{\delta}$--kinematically similar 
(with $\tilde{\delta}=\delta/ |b_{1}|$) to
\begin{displaymath}
z(n+1)=\diag(H_{1}(n),\ldots,H_{\ell}(n))\{I+R(n)\}z(n), 
\end{displaymath}
where 
\begin{displaymath}
H_{i}(n)=\diag\big(h_{i}(n),\ldots,h_{i}(n)\big) \quad \textnormal{and} \quad
R(n)=\diag\big(h_{1}^{-1}(n)R_{1}(n),\ldots,h_{\ell}^{-1}(n)R_{\ell}(n)\big)
\end{displaymath}
for any $i=1,\ldots,\ell$. Finally, note that $H(n)\subseteq \bigcup\limits_{i=1}^{\ell}[a_{i},b_{i}]=\Sigma(A)$ 
and $||R(n)||_{\infty}<\tilde{\delta}$ for any $n\in \mathbb{Z}$.

\emph{Final step:} Let $E$ a compact subset of $\mathbb{R}^{+}$ such that (\ref{LinA}) can be contracted to $E$.
By Proposition \ref{minimal}, we have that $\Sigma(A)\subseteq E$ and the minimality of $\Sigma(A)$
follows.
 
\subsection{Proof of Theorem \ref{Main2}} 
If (\ref{LinA}) has the full spectrum condition, Proposition \ref{bloques} says that (\ref{LinA})
is kinematically similar to 
\begin{displaymath}
y(n+1)=\diag\big(B_{1}(n),\ldots,B_{d}(n)\big)y(n). 
\end{displaymath}

As the matrix above has order $d\times d$, the results follows since is a diagonal matrix.

\appendix
\section{Proof of Lemma \ref{technical-result}}
We will construct a strictly increasing and unbounded sequence $\{N_{p}\}_{p=0}^{+\infty}$ satisfying $N_{0}=0$
such that the sequences $h_{i}$,$\Delta_{i}\colon \mathbb{Z}^{+}\cup \{0\}\to \mathbb{R}$ defined by:
\begin{displaymath}
h_{i}(j)=\left\{\begin{array}{rcl}
                     a_{j} &\textnormal{if} & j\in \{N_{l},\ldots,N_{l+1}-1\}, \quad (l=0,2,4,\ldots)\\\\
                     b_{j} &\textnormal{if} & j\in \{N_{l+1},\ldots,N_{l+2}-1\}
                    \end{array}\right.
\end{displaymath}
and
\begin{displaymath}
\Delta_{i}(j)=\left\{\begin{array}{rcl}
                     -\frac{\delta}{2\ell} &\textnormal{if} & j\in \{N_{l},\ldots,N_{l+1}-1\}, \quad (l=0,2,4,\ldots) \\\\
                     \frac{\delta}{2\ell} &\textnormal{if} & j\in \{N_{l+1},\ldots,N_{l+2}-1\},
                    \end{array}\right.
\end{displaymath}
satisfy properties (\ref{Cota21}) and (\ref{Cota22}) on $\mathbb{Z}^{+}\cup \{0\}$. The case for $\mathbb{Z}^{-}$ is similar
and be donde by the reader.

It is straightforward to see that (\ref{Cota21}) is always satisfied. In order to verify (\ref{Cota22}), we interchange $n$ by $k$ in 
the first inequality of (\ref{CotaKS}), we have:
\begin{equation}
\label{CotaKS2}
\left\{\begin{array}{rcl}
|U(n)U^{-1}(k)| \geq \theta^{k-n}  &\textnormal{if}& n\geq k\\\\
|S(n)S^{-1}(k)| \leq  \theta^{n-k}&\textnormal{if} & n\geq k.
\end{array}\right.
\end{equation}

By using induction, we will verify that for any $\mu>1$, there exists a sequence $\{N_{p}\}_{p}$ such that
\begin{equation}
\label{inegalite}
\frac{1}{\mu}\leq \prod\limits_{j=0}^{n-1}\Big|\frac{c_{i_{rr}}(j)}{h_{i}(j)+\Delta_{i}(j)}\Big|\leq \mu \quad \textnormal{for any $n\geq 0$}.
\end{equation}

If $k=0$ in the first inequality of (\ref{CotaKS2}) we obtain
\begin{displaymath}
|U(n)U^{-1}(0)| \geq \theta^{-n}  , \quad n\geq 0, 
\end{displaymath}
which implies that $|U(n)|$ is unbounded in $\mathbb{Z}^{+}$ since $\theta \in (0,1)$. In consequence, given $\mu>1$, there
exists $N_{1}>0$ such that $|U(n)U^{-1}(0)|<\mu$ for any $n\in \{0,\ldots,N_{1}\}$. In addition, by using the
first inequality of (\ref{CotaKS2}), have that
\begin{equation}
\label{Paso1a}
\frac{1}{\mu}<1\leq |U(n)U^{-1}(0)|<\mu, \quad \textnormal{$n\in \{0,\ldots,N_{1}\}$}.
\end{equation}

On the other hand, by using the second inequality of (\ref{CotaKS2}) and now considering $k=N_{1}$, we obtain that
the sequence $|U(N_{1})U^{-1}(0)||S(n)S^{-1}(N_{1})|$ is convergent to zero when $n\to +\infty$. Then, there exists 
$N_{2}>N_{1}$ such that  $|U(N_{1})U^{-1}(0)||S(n)S^{-1}(N_{1})|\geq 1/\mu$ for any $n\in \{N_{1},\ldots,N_{2}\}$. As before,
by using again the second inequality of (\ref{CotaKS2}) combined with (\ref{Paso1a}), we have that
\begin{equation}
\label{Paso1b}
\frac{1}{\mu}\leq |U(N_{1})U^{-1}(0)||S(n)S^{-1}(N_{1})|<\mu \quad \textnormal{$n\in \{N_{1},\ldots,N_{2}\}$}.
\end{equation}

If $n\in \{N_{1},\ldots,N_{2}\}$, by using our definition of $h_{i}(j)$ and $\lambda_{i}(j)$, we have that,
\begin{displaymath}
\begin{array}{rcl}
\prod\limits_{j=0}^{n-1}\Big|\frac{c_{i_{rr}}(j)}{h_{i}(j)+\Delta_{i}(j)}\Big|&= & \displaystyle \prod\limits_{j=0}^{N_{1}-1}\Big|\frac{c_{i_{rr}}(j)}{a_{i}-\frac{\delta}{2}}\Big| \prod\limits_{j=N_{1}}^{n-1}\Big|\frac{c_{i_{rr}}(j)}{b_{i}+\frac{\delta}{2}}\Big|\\\\
                                                                                     &= & \displaystyle |U(N_{1})U^{-1}(0)||S(n)S^{-1}(N_{1})|\\\\
                                                                                     \end{array}
\end{displaymath}
and (\ref{inegalite}) is verified for any $n\in \{0,\ldots,N_{2}\}$.

As Inductive hypothesis, we will assume that there exists $2m+1$ numbers \linebreak $0=N_{0}<N_{1}<N_{2}<\cdots<N_{2m-1}<N_{2m}$ such that
\begin{displaymath}
\frac{1}{\mu}\leq \prod\limits_{j=0}^{2m-1}|U(N_{j+1})U^{-1}(N_{j})||S(n)S^{-1}(N_{j+1})|\leq \mu \quad \textnormal{for any $n\in \{N_{2m-1},\ldots,N_{2m}\}$}.
\end{displaymath}

By using the first inequality of (\ref{CotaKS2}) and considering $n=2m$ in the inequality above, we have that 
\begin{equation}
\label{ultimopaso}
\Big(\prod\limits_{j=0}^{2m-2}|U(N_{j+1})U^{-1}(N_{j})||S(N_{j+2})S^{-1}(N_{j+1})|\Big)|U(n)U^{-1}(N_{2m})| 
\end{equation}
is unbounded for any $n>N_{2m}$. Then, there exists $N_{2m+1}$ such that this product is lower than $\mu$
for any $n\in \{N_{2m},\cdots,N_{2m+1}\}$. In addition, by inductive hypothesis combined with $|U(n)U^{-1}(N_{2m})|\geq 1$
(ensured by \ref{CotaKS2}), we can see that the product above is lowerly bounded by $1/\mu$.

Finally, we have that 
\begin{displaymath}
\Big(\prod\limits_{j=0}^{2m-1}|U(N_{j+1})U^{-1}(N_{j})||S(N_{j+2})S^{-1}(N_{j+1})|\Big)|S(n)S^{-1}(N_{2m+1})| 
\end{displaymath}
converges to zero when $n\to +\infty$. Then, there exists $N_{2m+2}$ such that  the product above is bigger than $1/\mu$.
As before, by (\ref{ultimopaso}) combined with $|S(n)S^{-1}(N_{2m+1})|\leq 1$ for any $n\geq N_{2m+1}$, we can deduce that
the product above is lower than $\mu$ and (\ref{inegalite}) is proved.

\end{document}